\documentclass{amsart}
\usepackage{amsfonts}
\usepackage{color}
\usepackage{graphicx}
\usepackage[bookmarksnumbered,colorlinks, linkcolor=blue, citecolor=red, pagebackref, bookmarks, breaklinks]{hyperref}

\newtheorem{thm}{Theorem}[section]
\newtheorem{cor}[thm]{Corollary}
\newtheorem{lema}[thm]{Lemma}
\newtheorem{prop}[thm]{Proposition}
\theoremstyle{definition}

\theoremstyle{remark}

\newtheorem{rem}[thm]{Remark}
\numberwithin{equation}{section}
\newcommand{\R}{\mathbb R}
\newcommand{\N}{\mathbb N}

\newcommand{\A}{\mathcal A}

\newcommand{\lam}{\lambda}
\newcommand{\cd}{\rightharpoonup}

\def\diver{\mathop{\text{\normalfont div}}}
\def\dist{\mathop{\text{\normalfont dist}}}

\begin{document}

\title{Lower bounds for Orlicz eigenvalues}
\author[A. M. Salort]
{Ariel M.  Salort}%

\address{Instituto de C\'alculo, FCEyN - Universidad de Buenos Aires and
\hfill\break \indent IMAS - CONICET
\hfill\break \indent Ciudad Universitaria, Pabell\'on I (1428) Av. Cantilo s/n. \hfill\break \indent Buenos Aires, Argentina.}

\email[A.M. Salort]{asalort@dm.uba.ar}
\urladdr{http://mate.dm.uba.ar/~asalort}


\begin{abstract}
In this article we consider the following  weighted nonlinear  eigenvalue problem for the $g-$Laplacian 
$$
-\diver\left( g(|\nabla u|)\frac{\nabla u}{|\nabla u|}\right) = \lam w(x) h(|u|)\frac{u}{|u|} \quad \text{ in }\Omega\subset \R^n, n\geq 1
$$
with Dirichlet boundary conditions. Here $w$ is a suitable weight and $g=G'$ and $h=H'$ are appropriated Young functions satisfying the so called $\Delta'$ condition, which includes  for instance logarithmic perturbation of powers and different power behaviors near zero and infinity. We prove several properties on its spectrum, being  our  main goal to obtain lower bounds of  eigenvalues in terms of $G$, $H$, $w$ and the normalization  $\mu$ of the corresponding eigenfunctions.

We introduce some new strategies to obtain results that generalize several inequalities from the literature of $p-$Laplacian type eigenvalues.

\end{abstract}

\maketitle

\section{Introduction}
The first eigenvalues of the one-dimensional $p-$Laplacian, $p>1$, with Dirichlet boundary condition,   i.e., $-(|u'|^{p-2}u')'=\lam |u|^{p-2}u$ in $\Omega=(a,b)$ with $u(a)=u(b)=0$,  can be computed explicitly as (see \cite{O})
$$
\lam_1 = \left(\frac{\pi_p}{b-a}\right)^p  \qquad \text{where }\pi_p :=2(p-1)^\frac1p \int_0^1 \frac{dt}{(1-t)^\frac1p}.
$$
However, in  the weighted version of the previous equation, i.e.,  $\lam w(x) |u|^{p-2}u$ in the right hand side for a suitable   function $w$,  an explicit  formula for $\lam_1(w)$ cannot be obtained. For this reason, to obtain lower bounds of this quantity has been a  subject of interest in the last years. As a consequence of a so-called Lyapunov type inequality (see \cite{Lya}), in \cite{E, LYHA,P,Pi, SL} are obtained bounds of the kin
\begin{align*}
\frac{\pi^p}{\|w\|_\infty (b-a)^p} &\leq \lam_1(w)\quad  \text{ when } 0\leq w\leq M \text{ in } (a,b),\\
\left(\frac{\pi_p}{\int_a^b w^\frac{1}{p}(x)\,dx}\right)^p &\leq \lam_1(w) \quad \text{ when }w\in L^1(a,b).
\end{align*}
 
When $\Omega\subset \R^n$, $n\geq 2$, explicit formulas for eigenvalues of the $p-$Laplacian are not available in general.  In \cite{AL} it is proved that the first eigenvalue of $-\Delta_p u:=-\diver(|\nabla u|^{p-2}\nabla u) =\lam w(x)|u|^{p-2}u$ in $\Omega$ with $u=0$ on $\partial\Omega$   satisfies that
\begin{equation*}
\frac{C}{\|w\|_{L^\infty(\Omega)} |\Omega|^\gamma} \leq \lam_1(w) \quad \text{ where } \gamma=\begin{cases}  \frac{p}{n} &\text{ when } 1<p\leq n,\\ \frac12 &\text{ when }n<p, \end{cases}
\end{equation*}
when $w\in L^\infty(\Omega)$. Similar bounds were obtained in \cite{C}:
\begin{equation*}
\frac{C}{\|w\|_{L^s(\Omega)} |\Omega|^\frac{sp-n}{sn}} \leq \lam_1(\Omega) \quad \text{ where } s=\begin{cases} s_0>\frac{n}{p} &\text{ when } 1<p\leq n\\ 1  &\text{ when }n<p.\end{cases}
\end{equation*}
In addition, in \cite{DNP}, bounds involving the inner radius $r_\Omega$ of $\Omega$ are obtained as a consequence of a Lyapunov type inequality:
\begin{align*}
\frac{C}{r_\Omega^{p-n} \|w^+\|_{L^1(\Omega)}} &\leq \lam_1(w) \quad\text{when}\quad p>n,\\ \frac{C}{r_\Omega^{\frac{sp-n}{s}} \|w^+\|_{L^s(\Omega)} }&\leq \lam_1(w) \quad \text{when}\quad p<n \text{ for any $s>0$ such that $sp>n$},
\end{align*}
where $C$ is a positive constant independent of $\lam_1(w)$, $w^+$ is the positive part of $w$ and $r_\Omega:=\max\{\dist(x,\partial\Omega)\colon x\in\Omega\}$. Later, in \cite{JKS} these results were extended to the nonlocal case.

In this manuscript we deal with the case of operators with a behavior more general than a power, whose prototype is  the $g-$Laplacian defined as $\Delta_g u:=\diver(g(|\nabla u|)\frac{\nabla u}{|\nabla u|})$, being $g$ the derivative of a Young function $G$ (see Section \ref{sec.young}). Here, the lack of homogeneity adds an extra difficulty and spoils up some of the techniques used in the powers case. See \cite{autov1,autov2}.  Now, formulas for eigenvalues are not available even in dimension one  without weight functions on the right hand side, and up to our knowledge, only a few papers have dealt with lower bounds of eigenvalues, and only for the one-dimension case. More precisely, in \cite{DNP-1d} some estimates  are obtained in terms of the upper and lower bounds of $G$, and in \cite{SV} some estimates for sub-multiplicative Young functions were proved.

Given Young functions $G$ and $H$ satisfying the so-called $\Delta_2$ condition (i.e., there is some positive constant $\bar C$ such that $G(2t)\leq \bar CG(t)$ for all $t\geq 0$, and the same for $H$), we consider the following  nonlinear  eigenvalue  problem in a bounded Lipschitz domain $\Omega\subset \R^n$, $n\geq 1$:
\begin{align} \label{eigen.intro}
\begin{cases}
-\diver\left( g(|\nabla u|)\frac{\nabla u}{|\nabla u|}\right) = \lam w(x) h(|u|)\frac{u}{|u|} &\text{ in } \Omega\\
u=0 &\text{ on } \partial\Omega,
\end{cases}
\end{align} 
where $\lam\in \R$ is the eigenvalue parameter, $w$ is a suitable positive weight function (typically, $w\in L^\infty(\Omega)$ or $w\in L^A(\Omega)$ for some Young function $A$), $g=G'$ and $h=H'$.

We first prove that under some suitable conditions on the growth behavior of $G$, $H$ and $w$, for each $\mu>0$ the quantity 
$$
\lam_{1,\mu}=\lam_{1,\mu}(\Omega,w):=\min\left\{ \frac{\int_\Omega G(|\nabla u|)\,dx}{\int_\Omega wH(|u|)\,dx} \colon u\in W^{1,G}_0(\Omega) \text{ and } \int_\Omega wH(|u|)\,dx =\mu\right\}
$$
is   the \emph{first eigenvalue} of \eqref{eigen.intro} with energy level $\mu$ (see Section \ref{sec.autov} for the precise definition of eigenvalue) in the sense that there is no $\lam\in (0,\lam_{1,\mu})$ with eigenfunction $u$ normalized such that $\int_\Omega wH(u)\,dx=\mu$ (see Proposition \ref{prop.lam1}) . Here $W^{1,G}_0(\Omega)$ is the natural Orlicz-Sobolev space where to study the problem. Moreover, with a further assumption on $G$ which guarantees the monotonicity of $\Delta_g$, in Proposition \ref{prop.closed} it is proved that the set of all eigenvalues of \eqref{eigen.intro} with normalized eigenfunctions is closed. Moreover, given $0<\mu_1<\mu_2$, in Proposition \ref{eig.mono} we prove that the  eigenvalues  are ordered as
$\mu_1 \lam_{1,\mu_1} \leq \mu_2 \lam_{1,\mu_2}$.

The main goal of this manuscript is to obtain quantitative lower bounds of $\lam_{1,\mu}(\Omega,w)$ in terms of $\Omega$  (involving quantities such as the volume of $\Omega$ or its inradius, i.e., the radius of the biggest ball contained in $\Omega$), $w$ and $\mu$. For this purpose we will consider the subclass of Young functions satisfying the so-called $\Delta'$ condition, i.e., there is a positive constant $C$ such that $G(ab)\leq CG(a)G(b)$ for every $a,b\geq 0$ (which implies the $\Delta_2$ condition). This class has as main examples $G(t)=t^p$ with $t\geq 0$ and $p>1$, $G(t)=t^p(1+|\log t|)$ with $t\geq 0$ and $p>1$ and $G(t)=\max\{t^p,t^q\}$ with $t\geq 0$ and $p,q>1$, among others (see Section \ref{sec.young} for more examples).

In our analysis will be fundamental the growth behavior of the Young functions, being $G^*$, the critical function in the Orlicz-Sobolev embedding (which plays the role of  $p^*=\frac{np}{n-p}$ in the case of powers),  the threshold function in our estimates. For this reason it is convenient to define the quantity $T_g$ and the \emph{critical Young function} $G^*$ as
$$
T_g:=\int_1^\infty \frac{G^{-1}(s)}{s^{1+\frac{1}{n}}}\,ds, \qquad (G^*)^{-1}(t)=\int_0^t  \frac{G^{-1}(s)}{s^{1+\frac{1}{n}}} \,ds.
$$
See Section \ref{sec.emb} for details.

We describe now our main results. In all the theorems $w$ denotes a non-negative function.

In Theorems \ref{teo.princ.1} we obtain lower bounds for $\lam_{1,\mu}(\Omega,w)$ when $\mu\geq 1$ in the case in which $T_g=\infty$ and $H$ and $G$ are ordered as $H\prec G \prec\prec G^*$:
\begin{align*}
\frac{1}{\mu}\left[ G\left( \frac{c_1}{ G^{-1}\left(\frac{1}{\tau_A(\Omega)\|w\|_{L^\infty(\Omega)} }\right)   }  \right) \right]^{-1}&\leq \lam_{1,\mu} \quad \text{ when } w\in L^\infty(\Omega),\\
\frac{1}{\mu}\left[ G\left( \frac{\hat c_1}{ G^{-1}\left(\|w\|^{-1}_{\tilde B}\right)   }  \right) \right]^{-1} &\leq \lam_{1,\mu}  \quad \text{ when } w\in L^{\tilde B}(\Omega) \text{ with } B\prec A,
\end{align*}
being $c_1$ and $\hat c_1$ computable positive constants independent of $\mu$; $A$ is the young function defined as $A=G^*\circ G^{-1}$ and  $\tau_A(\Omega)=|\Omega| A^{-1} (|\Omega|^{-1})$, where $\tilde A$ is conjugated Young  function of $A$.  The notation $H\prec G$ means that there exists a positive constant $k$ such that $H(t)\leq G(kt)$ for $t\geq t_0\geq 0$, and $G \prec\prec G^*$ means that  $G(t)/G^*(kt)\to 0$ as $t\to\infty$ for all $k>0$. Along this paper $\|\cdot\|_G$ denotes the so-called Luxemburg norm (see Section \ref{sec.young} for its definition).

\begin{rem}
In the case of the eigenvalue problem for the Dirichlet $p-$Laplacian, i.e.,  $G(t)=H(t)=t^p$, $p>1$,  Theorem  \ref{teo.princ.1} reads as
\begin{align*}
\frac{\bar c_1}{ \|w\|_{L^\infty(\Omega)}|\Omega|^\frac{p}{n}} &\leq \lam_{1} \qquad \text{ for } p<n \text{ and } w\in L^\infty(\Omega),\\
\frac{\bar c_1}{ \|w\|_{L^{s}(\Omega)}} &\leq \lam_{1} \qquad \text{ for } p<n \text{ and } w\in L^s(\Omega) \text{ with } sp>n.
\end{align*}
These results recover some of the lower bounds obtained in \cite{AL,C, DNP} in the case $n>p$.
\end{rem}

Our second result complements the previous inequality: in Theorem \ref{p.hardy} we prove that for any $\mu\geq 1$ there exists a positive constant $c_2$ independent of $\mu$ such that
$$
\frac{1}{\mu}\left[ G\left( \frac{ c_2 r_\Omega}{ H^{-1}(\|w\|_{L^\infty(\Omega)}^{-1})}\right) \right]^{-1} \leq \lam_{1,\mu} \quad \text{ when } w\in L^\infty(\Omega) \text{ and } \hat H \prec G,
$$
where $\hat H(t):= t\int_{t_0}^t \frac{H(s)}{s^2}\,ds$, $t\geq 0$ and $r_\Omega$ denotes the inner radius of $\Omega$.

\begin{rem}
In the case of the eigenvalue problem for the Dirichlet $p-$Laplacian, i.e.,  $G(t)=H(t)=t^p$, $p>1$,  Theorem  \ref{p.hardy} reads  as
$$
\frac{\bar c_2}{r_\Omega^p \|w\|_{L^\infty(\Omega)} } \leq \lam_1 \quad \text{ when } w\in L^\infty(\Omega) \text{ and }p> 1.
$$
In particular, since $B_{r_\Omega} \subset \Omega$ we have that $\omega_n r_\Omega^n \leq |\Omega|$ which gives
$$
\frac{\bar c_2}{|\Omega|^\frac{p}{n} \|w\|_{L^\infty(\Omega)} } \leq \lam_1 \quad \text{ when } w\in L^\infty(\Omega) \text{ and }p> 1.
$$
These results recover some of the lower bounds obtained in \cite{AL,C, DNP}.
\end{rem}
 
In Theorem \ref{p.bigger.n} we obtain lower bounds for $\lam_{1,\mu}(\Omega,w)$ when $\mu\geq 1$ in the case in which  $T_g<\infty$:
\begin{align*}
\frac{1}{\mu}\left[G\left(  \frac{ c_3\sigma(r_\Omega)}{H^{-1}(\|w\|_{L^1(\Omega)}^{-1})}\right)\right]^{-1} &\leq \lam_{1,\mu} \quad \text{ when } w\in L^\infty(\Omega)\\
\frac{1}{\mu}\left[G\left(\frac{\hat c_3\sigma(r_\Omega)\cdot \tau_H(\Omega)}{H^{-1}(\|w\|_{L^\infty(\Omega)}^{-1})}\right) \right]^{-1} &\leq \lam_{1,\mu}  \quad \text{ when } w\in L^\infty(\Omega),
\end{align*}
where $c_3$ and $\hat c_3$ are a computable positive constant independent of $\mu$, $\tau_H(\Omega):=|\Omega| (\tilde H)^{-1}(|\Omega|^{-1})$ and  $r_\Omega$ is the inner radius of $\Omega$. Here we have denoted
$$
\sigma(t)=\int_{t^{-n}}^\infty  \frac{G^{-1}(s)}{s^{1+\frac1n}}\,ds.
$$

\begin{rem}
In the case of the eigenvalue problem for the Dirichlet $p-$Laplacian, i.e.,  $G(t)=H(t)=t^p$, $p>1$,     Theorem \ref{p.bigger.n} reads as
\begin{align*}
\frac{\bar c_3}{r_\Omega^{p-n} \|w\|_{L^1(\Omega)}  } &\leq \lam_1\quad  \text{ when } w\in L^\infty(\Omega) \text{ and }  p>n,\\
\frac{\bar  c_3}{r_\Omega^{p-n} |\Omega|\|w\|_{L^\infty(\Omega)}  } &\leq \lam_1\quad  \text{ when } w\in L^\infty(\Omega) \text{ and }   p>n,
\end{align*}
where $\bar c_3$ is a positive constant independent of $|\Omega|$ and $r_\Omega$. These results recover some of results in \cite{DNP,Pi}.
\end{rem}

In the one-dimensional case, i.e., \eqref{eigen.intro} in a bounded interval $\Omega=(a,b)\subset \R$, more information can be retrieved. In Theorem \ref{teo1d-v1} it is proved that for any $\mu\geq 1$, 
$$
\frac{1}{\mu}\left[ G\left( 
c_4 (b-a) \frac{G^{-1}\left(\frac{1}{b-a}\right)}{H^{-1}\left(\|w\|_{L^1(\Omega)}^{-1}\right)}
 \right) \right]^{-1} \leq \lam_{1,\mu}  \quad \text{ when } w\in L^\infty(a,b)
$$
for any pair of Young functions $G$ and $H$ satisfying the $\Delta'$ condition, where $c_4$ is an explicit constant.

If we further assume that $H\prec G$ (i.e., $H(t)\leq G(kt)$ for $t\geq 0$ and for some $k\geq 0$), then in Theorem \ref{teo1d-v2} we prove that for any $\mu >0$
$$
\frac{b-a}{C k \|w\|_{L^1(\Omega)} G\left(\frac{k(b-a)}{2}\right)} \leq \lam_{1,\mu} \quad \text{ when } w\in L^\infty(a,b)
$$
where $C$ is the constant in the $\Delta'$ condition. In particular this gives that
$$
\left[ C k \|w\|_{L^\infty(\Omega)} G\left(\frac{k(b-a)}{2}\right) \right]^{-1} \leq \lam_{1,\mu} \quad \text{ when } w\in L^\infty(a,b).
$$
Observe that here there is no restriction on $\mu$.

\begin{rem}
In the case of the eigenvalue problem for the Dirichlet $p-$Laplacian, i.e.,  $G(t)=H(t)=t^p$, $p>1$, with $w\in L^\infty(\Omega)$ a positive weight function,   Theorem \ref{teo1d-v1} gives the following rough bounds
$$
\frac{1}{2^p \|w\|_{L^1(\Omega)} (b-a)^{p-1}} \leq \lam_1, \qquad
\frac{1}{2^p \|w\|_{L^\infty(\Omega)} (b-a)^{p}} \leq \lam_1,
$$ 
and Theorem \ref{teo1d-v2}  provides for the more accurate inequalities
$$
\frac{2^p}{\|w\|_{L^1(\Omega)} (b-a)^{p-1} }\leq \lam_1, \qquad \frac{2^p}{\|w\|_{L^\infty(\Omega)} (b-a)^p} \leq \lam_1,
$$
which recover   the known lower bound for the one-dimensional $p-$Laplacian obtained in \cite{LYHA,P,Pi, SL}.
\end{rem}

\subsection*{Organization of the paper}
The paper is organized as follows. In section  \ref{sec.prel} we introduce the notion of Young function  and some results regarding Orlicz-Sobolev spaces. Section \ref{sec.autov} is devoted to study some properties of the eigenvalue problem for the $g-$Laplacian with Dirichlet boundary conditions. In section \ref{sec.teo} we prove our main theorems in $\R^n$ while in section \ref{sec.1d}, we deliver for the proofs in the one-dimensional case. Finally in section \ref{sec.rem} we mention some remarks and final comments.

\section{Preliminaries} \label{sec.prel}

Along this article $\Omega$ will denote an open and bounded set in $\R^n$, $n\geq 1$ with Lipschitz boundary. The \emph{inner radius} of $\Omega$ is defined as
$r_\Omega:=\max\{\dist(x,\partial\Omega)\colon x\in\Omega\}$, where
$$
\dist(x,\partial\Omega):=\inf\{|x-y|\colon y\in\partial\Omega\}
$$
is the distance from $x\in\Omega$ to the boundary $\partial\Omega$.

\subsection{Young functions} \label{sec.young}

An application $G\colon[0,\infty)\longrightarrow [0,\infty)$ is said to be a  \emph{Young function} if it admits the integral representation 
\[
G(t)=\int_0^t g(s)\,ds,
\] 
where the right-continuous function $g$ defined on $[0,\infty)$ has the following properties:
\begin{itemize}
\item[(i)] $g(0)=0$, \quad $g(t)>0 \text{ for } t>0$, 
\item[(ii)]  $g \text{ is nondecreasing on } (0,\infty)$, 
\item[(iii)]  $\lim_{t\to\infty}g(t)=\infty$.
\end{itemize}

From these properties it is easy to see that a Young function $G$ is continuous, nonnegative, strictly increasing and convex on $[0,\infty)$. 

Without loss generality we can assume $G(1)=1$.
 
We will consider the class of Young functions such that $g=G'$ is an absolutely continuous function that satisfies the condition 
\begin{equation} \label{condp}
1\leq p^-\leq \frac{tg(t)}{G(t)}\leq p^+<\infty, \qquad t\geq 0.
\end{equation}
Condition \eqref{condp} is equivalent to ask $G$ and $\tilde G$ to satisfy the \emph{$\Delta_2$ condition or doubling condition}, i.e., 
$$
G(2t)\leq 2^{p^+} G(t), \qquad \tilde G(2t) \leq t^{p^-} \tilde G(t),
$$
where the \emph{complementary} function of a Young function $G$ is the Young function defined as
$$
\tilde G(t)=\sup\{ta-G(a)\colon a>0\}.
$$

We say that the Young function $G$ satisfies the \emph{$\Delta'-$condition} if there exists a positive constant $C\geq 1$ such that
\begin{equation} \label{cond}
G(ab)\leq C G(a)G(b)
\end{equation}
for all $a,b\geq 0$.

In particular, if a Young function $G$ satisfies $\Delta'$ condition, then it satisfies the $\Delta_2$ condition.

Observe that taking $a=G^{-1}(s)$, $b=G^{-1}(t)$, for some $s,t\geq 0$, \eqref{cond} gives
$$
G(G^{-1}(s) G^{-1}(t)) \leq C G(G^{-1}(s))G(G^{-1}(t))=C st
$$
from where
\begin{equation} \label{cond1}
G^{-1}(s) G^{-1}(t) \leq  c G^{-1}(st)
\end{equation}
for all $s,t\geq 0$, where $c=c(C,p^+,p^-)$.

From now on, $C$ and $c$ will denote the constants in \eqref{cond} and \eqref{cond1}.

Examples of Young functions satisfying the $\Delta'$ condition include:
\begin{itemize}
\item $G(t)=t^p$, $t\geq0$, $p>1$;
\item $G(t)=t^p (1+|\log t |)$, $t\geq0$, $p>1$;
\item $G(t)= t^p \chi_{(0,1]}(t) + t^q \chi_{(1,\infty)}(t)$, $t\geq0$, $p,q>1$;
\item $G(t)$ given by the complementary function to $\tilde G(t)=(1+t)^{\sqrt{\log(1+t)}}-1$, $t\geq0$.
\end{itemize}

We conclude this brief introduction by mentioning some useful properties that Young functions fulfill.

\begin{lema} \label{monot}
Assume that $G$ is a Young function satisfying
\begin{equation} \label{condp'}
1\leq p^- -1 \leq \frac{tg'(t)}{g(t)}\leq p^+-1<\infty, \qquad t\geq 0
\end{equation}
for some finite constants $p^\pm>2$. Then for all $a,b\in \R^n$ it holds that
\begin{equation*}
\left(\frac{g(|a|)}{|a|}a-\frac{g(|b|)}{|b|}b \right)\cdot (a-b) \geq 0.
\end{equation*}
\end{lema}

\begin{rem}
Observe that condition \eqref{condp'} is the equivalent to assume $p\geq 2$ in the case $G(t)=t^p$, and it implies in particular condition \eqref{condp}.
\end{rem}

Given a Young function $G$, $\Omega\subset \R^n$ and a non-negative measurable and locally integrable function $w$, we define the \emph{modular} $\Phi_{G,w}$ as
$$
\Phi_{G,w}(u) =\int_\Omega G(|u|) w(x)\,dx.
$$
The so-called \emph{Luxemburg} norm is defined as
$$
\|u\|_{G,w} := \inf\left\{ \delta>0\colon \Phi_{G,w}\left(\frac{u}{\delta}\right) \leq 1 \right\}.
$$

When $w\equiv 1$ we just write $\Phi_G(\cdot)$ and $\|\cdot\|_G$.

\begin{lema} \cite[Lemma 2.1.14]{DHHR} \label{equiv}
Let $G$ be a Young function. Then 
\begin{itemize}
\item[(i)] $\Phi_{G,w}(u)\leq 1 \iff \|u\|_{G,w} \leq 1$,
\item[(ii)] $\Phi_{G,w}(u)< 1 \iff \|u\|_{G,w} < 1$,
\item[(iii)] $\Phi_{G,w}(u)= 1 \iff \|u\|_{G,w} = 1$.
\end{itemize}
\end{lema}

Finally, given a Young function $G$, the H\"older's inequality for Young functions reads as (see for instance \cite[Lemma 2.6.5]{DHHR})
$$
\int_\Omega |uv| w(x)\,dx \leq 2\|u\|_{G,w} \|v\|_{\tilde G,w},
$$
where $\tilde G$ is the complementary function of $G$, and $\Omega\subset \R^n$.

\subsection{Orlicz-Sobolev spaces}

Given a Young function $G$, an open set $\Omega\subset \R^n$ and a non-negative measurable and locally integrable function $w$ we define
\begin{align*}
L^{G,w}(\Omega) &:=\{u\colon \R^n\to \R \text{ measurable such that }\Phi_{G,w}(u)<\infty\},\\
W^{1,G}(\Omega)&:=\{ u\in L^G(\Omega) \text{ such that }\Phi_{G}(|\nabla u|)<\infty\},
\end{align*}
where $\nabla u$ is considered in the distributional sense. When $w\equiv 1$ we just write $L^G$. 
These spaces are endowed, respectively,  with the Luxemburg norms $\|u\|_{G,w}$
and
$$
\|u\|_{1,G}=\|u\|_G + \|\nabla u\|_G.
$$

In light of \eqref{condp}, $L^{G,w}(\Omega)$ and $W^{1,G}(\Omega)$ are Banach and separable spaces.

We denote $W^{1,G}_0(\Omega)$ the closure of $C_c^\infty(\Omega)$ in $W^{1,G}(\Omega)$. When $\Omega$ is a Lipschitz domain we have the following characterization
$$
W^{1,G}_0(\Omega):=\{ u \in W^{1,G}(\Omega) \colon u=0 \text{ on } \partial\Omega\}.
$$
Then, due to the Poincar\'e's inequality for $u\in W^{1,G}_0(\Omega)$ we get
$$
\|u\|_G \leq C_p \| \nabla u \|_G,
$$
and then $\| \nabla u \|_G$ is an equivalent norm in $W^{1,G}_0(\Omega)$.

\subsection{Some embedding results} \label{sec.emb}

Let $G$ and $H$ be Young functions. If there exists a positive constant $k$ such that $G(t)\leq H(kt)$ for $t\geq t_0\geq 0$ we write $G\prec H$. Moreover, we write $G\prec\prec H$ when for all $k>0$ it holds that
$$
\lim_{t\to\infty} \frac{G(t)}{H(kt)}=0.
$$
It is obvious that $G\prec\prec H$ implies that $G\prec H$, but the converse does not hold. Moreover, $G \prec H$ if and only if $L^H(\Omega) \subset L^G(\Omega)$ (see for instance \cite[Theorem 3.17.1]{CFK}).

Define
\begin{equation*} 
T_g:=\int_1^\infty \frac{G^{-1}(t)}{t^{1+\frac{1}{n}}}\,dt, \qquad\text{ and } \qquad (G^*)^{-1}(t)=\int_0^t \frac{G^{-1}(s)}{s^{1+\frac{1}{n}}}\,ds, \; t\geq 0,
\end{equation*}
where $G^*$ is the so-called \emph{critical} Young function for the Young function $G$.
Then the following embedding results hold. See for instance \cite{Adams, CFK}.

\begin{prop} \label{embed}
Let $\Omega\subset\R^n$ be a bounded Lipschitz domain and let $G$ and $H$ be two Young functions.

If $T_g=\infty$  and $H\prec\prec G^*$, then the following embedding is compact
$$
W^{1,G}(\Omega)\subset L^{H}(\Omega).
$$ 
In particular, when $u\in W^{1,G}_0(\Omega)$ there is $\kappa=\kappa(n,p^+,p^-)$ such that
$$
\|u\|_{H} \leq \kappa \| \nabla u\|_G.
$$

If $T_g<\infty$ then 
$$
W^{1,G}(\Omega)\subset C^{0,\sigma(t)}(\overline \Omega)
$$
where $\sigma(t)$ is an increasing continuous function for $t\geq 0$ such that $\sigma(0)=0$ and it is given by
\begin{equation} \label{modulo}
\sigma(t)=\int_{t^{-n}}^\infty  \frac{G^{-1}(s)}{s^{1+\frac1n}}\,ds.
\end{equation}
In particular, when $u\in W^{1,G}_0(\Omega)$  there is $\kappa=\kappa(n,p^+,p^-)$ such that
\begin{equation} \label{des1}
\frac{|u(x)-u(y)|}{\sigma(|x-y|)}\leq \kappa \|\nabla u\|_{G}.
\end{equation}
Moreover, we have also that the   embedding $W^{1,G}(\Omega)\subset L^{H}(\Omega)$ is compact.
\end{prop}

The following result states some embeddings for weighted Orlicz spaces  which are useful for our purposes.
\begin{prop}  \label{embed.w}
Let $\Omega\subset\R^n$ be a bounded Lipschitz domain and let $G$ and $H$ be Young functions satisfying condition \eqref{cond}.

If $T_g=\infty$  and $H\prec\prec G^*$, then the following embedding 
$$
W^{1,G}(\Omega)\subset L^{H,w}(\Omega)
$$ 
is compact for 
\begin{itemize}
\item [(i)] $w\in L^\infty(\Omega)$;
\item[(ii)] more generally, for $w\in L^{\tilde B}(\Omega)$, with $B\prec A$ where $\tilde B$ is the conjugated function of $B$ and $A$ is the Young function given by $A=G^*\circ G^{-1}$.

\end{itemize}

If $T_g<\infty$, then 
$$
W^{1,G}(\Omega)\subset C^{0,\sigma(t)}(\overline \Omega)
$$
where $\sigma$ is given in \eqref{modulo}. Moreover, the embedding 
$$
W^{1,G}(\Omega)\subset L^{H,w}(\Omega)
$$ 
is compact for $w\in L^\infty(\Omega)$.
\end{prop}

\begin{proof}

Let us assume first that $w\in L^\infty(\Omega)$. Given given $u\in W^{1,G}(\Omega)\cap L^H(\Omega)$ define the number 
$$
\delta=\frac{\|u\|_H}{H^{-1}(C^{-1}\|w\|_\infty^{-1})}
$$
being $C$ the constant of condition \eqref{cond}. Then, the  definition of the Luxemburg norm together with \eqref{cond}  yields
\begin{align*}
\int_\Omega H\left( \frac{|u|}{\delta} \right) w(x)\,dx &\leq  
\int_\Omega H\left( \frac{|u|}{\|u\|_H} H^{-1}(C^{-1}\|w\|_\infty^{-1})\right) \|w\|_\infty \,dx\\  &\leq 
\int_\Omega H\left( \frac{|u|}{\|u\|_H}\right) \,dx =1,
\end{align*}
from where, again by definition of the Luxemburg norm, we get
$$
\|u\|_{H,w}\leq \frac{\|u\|_H}{H^{-1}(C^{-1}\|w\|_\infty^{-1})}.
$$
which proves that $L^H(\Omega)\subset L^{H,w}(\Omega)$. Then the result follows in light of  Proposition \ref{embed}.

Assume now that $w\in L^{\tilde A}(\Omega)$.
Observe that  it always holds that $G\prec \prec G^*$ (see for instance \cite[Example 6.3]{DSSSS}). Moreover,    since $A\circ G=G^*$ is a Young function, by \cite{KR}[Chapter I.1.2], $A$ is also a Young function.

Let $u\in W^{1,G}(\Omega)$ and let us define the number
$$
\delta= \frac{c k \|u\|_{G^*}}{ G^{-1}\left( \frac{1}{2    C \|w\|_{\tilde A}}\right) }
$$
where $k$ is the constant for which $H\prec G$,   and $C$ is given in \eqref{cond}. By using \eqref{cond} and H\"older's inequality for Young functions we get
\begin{align*}
\int_\Omega H\left( \frac{|u|}{\delta}\right) w(x)\,dx &\leq  \int_\Omega  G\left( \frac{k|u|}{\delta}\right) w(x)\,dx\\
&\leq \frac{1}{2\|w\|_{\tilde A}} \int_\Omega  G\left( \frac{|u|}{c\|u\|_{G^*}} \right)  w(x)\,dx\\
&\leq \frac{1}{\|w\|_{\tilde A}} \cdot \left\| G\left( \frac{|u|}{c\|u\|_{G^*}} \right) \right\|_A \|w\|_{\tilde A} = \left\| G\left( \frac{|u|}{c\|u\|_{G^*}} \right) \right\|_A.
\end{align*}
To estimate the last norm denote $v=\frac{u}{c\|u\|_{G^*}}$. Then,  by using \eqref{cond1}
\begin{align*}
G^{-1}\left( \frac{G(|v|)}{ G(c\|v\|_{G^*})} \right) &= 
\frac{G^{-1}(G(c\|v\|_{G^*}))}{ c\|v\|_{G^*} } G^{-1}\left( \frac{G(|v|)}{ G(c\|v\|_{G^*})} \right)\\
&\leq
 \frac{1}{ \|v\|_{G^*} } G^{-1}\left(G(c\|v\|_{G^*})  \frac{G(|v|)}{G(c\|v\|_{G^*})}  \right) =
 \frac{ |v|}{ \|v\|_{G^*} }.
\end{align*}
As a consequence, the definition of the Luxemburg norm gives
\begin{align*}
\int_\Omega A\left( \frac{G(|v|)}{G(c\|v\|_{G^*})} \right)\,dx &= 
\int_\Omega G^*\circ G^{-1}\left( \frac{G(|v|)}{ G(c\|v\|_{G^*})} \right)\,dx\\
&\leq 
\int_\Omega G^*\left(  \frac{|v|}{ \|v\|_{G^*} }  \right) \,dx=1,
\end{align*}
and using again the definition of the norm,  we arrive at
$$
\|G(|v|) \|_A  \leq G(c\|v\|_{G^*})
$$
that is, 
$$
\left\| G\left( \frac{|u|}{c\|u\|_{G^*}} \right) \right\|_A    \leq  G\left( c \left\| \frac{u}{c\|u\|_{G^*}}\right\|_{G^*} \right)=1.
$$
The previous inequalities lead to
$$
\int_\Omega H\left( \frac{|u|}{\delta}\right) w(x)\,dx \leq \left\| G\left( \frac{|u|}{c\|u\|_{G^*}} \right) \right\|_A =1
$$
since the norm is $1-$homogeneous. So, by the definition of Luxemburg norm
$$
\|u\|_{H,w} \leq \delta =  \frac{c k \|u\|_{G^*}}{ G^{-1}\left( \frac{1}{2    C \|w\|_{\tilde A}}\right) }.
$$
Then the result in the case $w\in L^{\tilde A}(\Omega)$   follows by Proposition \ref{embed}. The case $w\in L^{\tilde B}(\Omega)$ with $B\prec A$ is analogous. The proof is now complete.
\end{proof}

\section{The nonlinear eigenvalue problem} \label{sec.autov}
Given Young functions $G$ and $H$ satisfying \eqref{cond}  denote $h=H'$, $g=G'$. When $T_g=\infty$  we further assume that $H\prec\prec G^*$. In this section we deal with the following problem
\begin{align} \label{eigen}
\begin{cases}
-\diver\left( g(|\nabla u|)\frac{\nabla u}{|\nabla u|}\right) = \lam w(x) h(|u|)\frac{u}{|u|} &\text{ in } \Omega,\\
u=0 &\text{ on } \partial\Omega,
\end{cases}
\end{align}
where $\lam\in\R$ is a parameter and $\Omega\subset \R^n$ $n\geq 1$ is a bounded domain with Lipschitz boundary. We assume that $w$ is a non-negative weight function with the following integrability: when $T_g<\infty$,  then $w\in L^\infty(\Omega)$; when $T_g=\infty$ and $H\prec\prec G^*$, then $w\in L^\infty(\Omega)$, or, more generally,  $w\in L^{\tilde B}(\Omega)$, with $B\prec A:=G^*\circ G^{-1}$.

In order to study \eqref{eigen} we consider the functionals $I$, $J\colon W^{1,G}_0(\Omega)\to \R$ defined as
$$
I(u)=\int_\Omega G(|\nabla u|)\,dx, \qquad J(u)=\int_\Omega wH(|u|)\,dx.
$$
It is easy to see that $I$ and $J$ are Fr\'echet derivable with derivatives given by
$$
\langle I'(u),v \rangle =  \int_\Omega g(|\nabla u|)\frac{\nabla u}{|\nabla u|}\cdot \nabla v \,dx, \qquad 
\langle J'(u),v \rangle =  \int_\Omega w h(|u|) \frac{uv}{|u|}\,dx.
$$

Given $\mu>0$ we define the set $\mathcal{A}_\mu$ of normalized functions as
$$
\A_\mu:=\left\{u\in W^{1,G}_0(\Omega)\colon \int_\Omega w H(|u|)\,dx =\mu\right\}.
$$
Then, we say that $\lam$ is an eigenvalue of \eqref{eigen} with eigenfunction $u\in \A_\mu$ if 
$$
\langle I'(u),v \rangle = \lam  \langle J'(u),v \rangle
$$
for all $v\in W^{1,G}_0(\Omega)$.

The set of all eigenvalues with  normalized eigenfunctions is denoted as
$$
\Sigma_\Lambda=\{\lam\in \R \colon \lam \text{ is an eigenvalue of }\eqref{eigen} \text{ with eigenfunction } u\in \mathcal{A}_\mu \text{ for some } \mu\leq \Lambda  \}.
$$

\begin{prop} \label{prop.lam1}
For each $\mu>0$, the quantity
\begin{equation} \label{eigen.variac}
\lam_{1,\mu}:=\inf\left\{\frac{I(u)}{J(u)}\colon u\in \A_\mu \right\}.
\end{equation}
is an eigenvalue of \eqref{eigen}.

The corresponding eigenfunction $u_{1,\mu}\in \A_\mu$ may be assumed to be strictly   positive in $\Omega$ and $u_{1,\mu} \in C^{1,\alpha}(\overline \Omega)$ for some $\alpha\in (0,1)$.
\end{prop}

\begin{rem}
We call to \eqref{eigen.variac}  \emph{first eigenvalue} at level $\mu$ since there is no variational eigenvalue $\lam \in (0,\lam_{1,\mu})$ with eigenfunction in $\A_\mu$.  Indeed, if there is $\lam\in(0,\lam_{1,\mu})$ such that $\lam=I(u)/J(u)$ for some $u\in \A_\mu$, then it would be $\lam<\lam_1$ by definition of $\lam_{1,\mu}$, a contradiction.
\end{rem}

\begin{proof}
Let $\{v_k\}_{k\in\N}\in \A_\mu$ be a minimizing sequence in $\A_\mu$ for $\lam_{1,\mu}$, that is, if we denote $\mathcal{K}:=\frac{I}{J}$, 
$$
\lim_{k\to\infty} \mathcal{K}(v_k)=\lam_{1,\mu} = \inf_{u\in \A_\mu} \mathcal{K}(u).
$$
Then $\mathcal{K}(v_k)=\mu^{-1}I(v_k)\leq 1+\lam_{1,\mu}$ for $k$ big enough, from where $\{v_k\}_{k\in\N}$ is bounded in $W^{1,G}_0(\Omega)$. Therefore, due to Proposition \ref{embed.w}, up to a subsequence,  there is a function  $u\in W^{1,G}_0(\Omega)$ such that
$$
v_k\cd u \text{ weakly in }W^{1,G}_0(\Omega), \quad 
v_k\to u \text{ strongly in }L^{H,w}(\Omega) \text{ and a.e. }
$$ 
which gives that $u\in \mathcal{A}_\mu$. This yields
$$
\lam_{1,\mu} \leq \mathcal{K}(u) = \mu^{-1}I(u) \leq \mu^{-1} \liminf_{k\to\infty} I(v_k) = \lim_{k\to\infty} \mathcal{K}(v_k)=\lam_{1,\mu},
$$
where we have used the lower semicontinuity of the modular with respect to the weak convergence. This proves that $u$ attains the infimum in \eqref{eigen.variac}.

Since $I$ and $J$ are differentiable we have
\begin{align*}
0&=\langle \mathcal{K}'(u),v \rangle=\langle \frac{I'(u) J(u)}{I(u)I(u)},v \rangle-\langle \frac{J'(u)}{I(u)},v \rangle \\
&=\frac{1}{\lam_{1,\mu}}\langle \frac{I'(u) }{I(u)},v \rangle-\langle \frac{J'(u)}{I(u)},v \rangle  \qquad \forall v\in W^{1,G}_0(\Omega)
\end{align*}
from where $\langle I'(u),v \rangle = \lam_{1,\mu}  \langle J'(u),v \rangle$ $\forall v\in W^{1,G}_0(\Omega)$ and then $\lam_{1,\mu}$ is an eigenvalue of \eqref{eigen}.

If $u_{1,\mu}\in \mathcal{A}_\mu$ is an eigenfunction corresponding to $\lam_{1,\mu}$, then $u_{1,\mu}$ is strictly positive (or strictly negative) in $\Omega$ by the Strong Maximum Principle (see for instance \cite[Lemma 6.4]{L}). Moreover, due to the regularity  theory developed in \cite{L}, $u_{1,\mu}\in L^\infty(\Omega)$ and therefore it belongs to $C^{1,\alpha}(\bar \Omega)$ for some $\alpha\in (0,1)$.
\end{proof}

\begin{rem} \label{rem.autov}
In particular, if $u_{1,\mu}\in \A_\mu$ is a function attaining the  minimum in \eqref{eigen.variac} it follows that
$$
I(u_{1,\mu})=\lam_{1,\mu} J(u_{1,\mu}).
$$
\end{rem}

We prove now that the eigenvalues defined in \eqref{eigen.variac} are ordered according their normalization as follows.

\begin{prop} \label{eig.mono}
Given $0<\mu_1<\mu_2$ we have that
$$
\mu_1\lam_{1,\mu_1} \leq \mu_2 \lam_{1,\mu_2}.
$$
\end{prop}

\begin{proof}
Let $0<\mu_1<\mu_2$ be fixed numbers and consider the eigenvalues $\lam_{1,\mu_i}$  with corresponding eigenfunctions $u_{1,\mu_i} \in \A_{\mu_1}$, $i=1,2$. By Proposition \ref{prop.lam1} $u_{1,\mu_i}$ can be assumed to be positive in $\Omega$, $i=1,2$.

Since $H$ is increasing, $w$ is positive, $u_{1,\mu_2}>0$ in $\Omega$ and $\mu_2>\mu_1$, there exists $\kappa=\kappa(\mu_1,\mu_2)\in (0,1)$ such that
$$
\mu_2 = \int_\Omega wH(u_{1,\mu_2})\,dx > \int_\Omega wH( \kappa u_{1,\mu_2}) \,dx=\mu_1.
$$
Then the function $\kappa u_{1,\mu_2} \in \A_{\mu_1}$ can be taken as a test for $\lam_{1,\mu_1}$ to get
$$
\mu_1 \lam_{1,\mu_1} \leq \int_\Omega G( \kappa|\nabla u_{\mu_2} |)\,dx \leq  \int_\Omega G(|\nabla u_{\mu_2}|)\,dx= \mu_2 \lam_{1,\mu_2},
$$
giving the result.
\end{proof}

\begin{cor} \label{coro}
Given $\mu\geq 1$ we have that $\lam_{1,\mu} \geq \frac{\lam_{1,1}}{\mu}$.
\end{cor}

Finally we prove the closedness of the set $\Sigma_\Lambda$ for each $\Lambda>0$.

\begin{prop} \label{prop.closed}
Further assume that $G$ satisfies \eqref{condp'}. Given $\Lambda>0$, then $\Sigma_\Lambda$ is a closed set.
\end{prop}
\begin{proof}
First, since $G(t)$, $H(t)$ and $w(t)$ are positive functions for all $t>0$, it is clear that eigenvalues of \eqref{eigen} are positive.

Let $\lam_k\in \Sigma_\Lambda$ be such that $\lam_k\to \lam$ and let $u_k\in \mathcal{A}_{\mu_k}$, with $\mu_k \leq \Lambda$,  be an eigenfunction associated to $\lam_k$. Then, since for $k$ large enough, $\lam_k$ is bounded, by definition of eigenvalue and \eqref{condp} we get
\begin{align*}
\int_\Omega G(|\nabla u_k|)\,dx  &\leq \frac{1}{p^-} \int_\Omega g(|\nabla u_k|) |\nabla u_k|\,dx = \frac{\lam_k}{p^-} \int_\Omega w h(|u_k|)|u_k|\,dx\\
&\leq \lam_k \frac{p^+}{p^-} \int_\Omega w H(|u_k|)\,dx \leq \Lambda (1+\lam)\frac{p^+}{p^-}
\end{align*}
and hence $\{u_k\}_{k\in\N}$ is bounded in  $W^{1,G}_0(\Omega)$. Then, by using Proposition \ref{embed.w}, up to some subsequence if necessary, there exist $\mu\leq \Lambda$,  a function $u\in W^{1,G}_0(\Omega)$ and $\eta \in (L^{\widetilde G})^n$ such that
\begin{align*}
u_k\cd u &\text{ weakly in }W^{1,G}_0(\Omega),\\
u_k \to u &\text{ strongly in }L^{H,w}(\Omega) \text{ and a.e. in }\Omega,\\
g(|\nabla u_k|)\frac{\nabla u_k}{|\nabla u_k|}\cd \eta &\text{ weakly in }(L^{\widetilde G}(\Omega))^n.
\end{align*}
From this, $\int_\Omega wH(|u|)\,dx=\mu$, and then $u\neq 0$, and
\begin{equation} \label{eqnu}
\int_\Omega \eta \cdot \nabla v\,dx = \lam \int_\Omega w h(|u|)\frac{uv}{|u|}\,dx
\end{equation}
for every $v\in W^{1,G}_0(\Omega)$. As a consequence, the proof finishes if we show that
\begin{equation} \label{aprobar}
\int_\Omega \eta  \cdot  \nabla v\,dx = \int_\Omega g(|\nabla u|)\frac{\nabla u}{|\nabla u|}\cdot \nabla v\,dx
\end{equation}
for every $v\in W^{1,G}_0(\Omega)$. For this purpose we make use of the monotonicity of $g$ (Lemma \ref{monot}). For every $r\in W^{1,G}_0(\Omega)$, $u_k-r\in W^{1,G}_0(\Omega)$ 
and 
\begin{align*}
0&\leq \int_\Omega (g(|\nabla u_k|)\frac{\nabla u_k}{|\nabla u_k|}-g(|\nabla r|)\frac{\nabla r}{|\nabla r|})\cdot(\nabla u_k -\nabla r)\,dx\\
&=\lam_k \int_\Omega w h(|u_k|)\frac{u_k}{|u_k|}(u_k-r)\,dx - \int_\Omega g(|\nabla r|)\frac{\nabla r}{|\nabla r|}\cdot(\nabla u_k -\nabla r)\,dx.
\end{align*}
Taking limit $k\to\infty$ in the previous inequality, and using \eqref{eqnu}, we get
\begin{align*}
0&\leq \lam \int_\Omega w h(|u|)\frac{u}{|u|}(u-r)\,dx - \int_\Omega g(|\nabla r|)\frac{\nabla r}{|\nabla r|}\cdot(\nabla u -\nabla r)\,dx\\
&=\int_\Omega \eta \cdot (\nabla u-\nabla r)\,dx - \int_\Omega g(|\nabla r|)\frac{\nabla r}{|\nabla w|}\cdot(\nabla u -\nabla r)\,dx.
\end{align*}
So, if we take $r=u-tv$, with $v\in W^{1,G}_0(\Omega)$,  given $t> 0$, we immediately get
$$
0\leq \int_\Omega \left(\eta -  g(|\nabla u -t\nabla v|)\frac{\nabla u - t \nabla v}{|\nabla u - t\nabla v|} \right)\cdot \nabla v\,dx
$$
and taking $t\to 0^+$ we arrive at
$$
0\leq \int_\Omega \left(\eta -  g(|\nabla u|)\frac{\nabla u}{|\nabla u|} \right)\cdot \nabla v\,dx.
$$
In a analogous way, interchanging the roles of $u_k$ and $r$, it can be obtained the opposite inequality and then \eqref{aprobar} follows.
\end{proof}

\section{Lower bound of eigenvalues} \label{sec.teo}

Throughout this section $G$, $H$, $w$ and $\Omega$ satisfy the hypothesis introduced for \eqref{eigen}.

We  denote $\lam_1:=\lam_{1,1}$, that is, the eigenvalue \eqref{eigen.variac} with energy level $\mu=1$. Similarly, we denote $\A:=\A_{1}$.

In light of Corollary \ref{coro}, for any $\mu\geq 1$, we have that a lower bound for $\lam_1$ also bounds by below to $\mu \lam_{1,\mu}$. For this reason,  we focus on lower bounds for $\lam_1$.

\begin{prop} \label{prop1}
Let $\lam_1$ be the first eigenvalue of \eqref{eigen} with eigenfunction $u\in \A$, then
$$
\|   \nabla u   \|_G \leq \frac{\kappa_0}{G^{-1}(\lam_1^{-1}) } \|u\|_{H,w},
$$
where $\kappa_0=C^{p^+}$ is a positive constant.
\end{prop}

\begin{proof}
Since $\lam_1$ is an eigenvalue of \eqref{eigen} with eigenfunction $u\in \A$ we get
$$
\frac{1}{\lam_1}\int_\Omega G(|\nabla u|)\,dx=\int_\Omega wH(|u|)\,dx =1,
$$
from where, using \eqref{cond} and \eqref{condp} it follows that
\begin{align*}
  1&=\frac{1}{\lam_1}\int_\Omega G(|\nabla u|)\,dx= G\left(G^{-1}(\lam_1^{-1})\right) \int_\Omega G(|\nabla u|)\,dx\\
  &\geq \frac{1}{C} \int_\Omega G \left(G^{-1}(\lam_1^{-1}) |\nabla u| \right)\,dx\\
  &\geq  \int_\Omega G \left( \kappa_0^{-1} G^{-1}(\lam_1^{-1}) |\nabla u| \right)\,dx
\end{align*}
where $\kappa_0=C^{p^+}$.

Therefore, the last inequality together with Lemma \ref{equiv} yields
$$
\| \kappa_0^{-1} G^{-1}(\lam_1^{-1}) |\nabla u|  \|_G \leq 1,
$$
and since the norm is $1-$homogeneous
$$
\|  \nabla u   \|_G \leq \frac{\kappa_0}{ G^{-1}(\lam_1^{-1}) }.
$$

Again, the normalization $\Phi_{H,w}(u)=1$ implies $\|u\|_{H,w}=1$ due to Lemma \ref{equiv},  and then
$$
\|   \nabla u   \|_G \leq \frac{\kappa_0}{ G^{-1}(\lam_1^{-1}) } \|u\|_{H,w}
$$
which concludes the proof.
\end{proof}

\begin{thm} \label{teo.princ.1}
Let $G$ and $H$ be Young functions  satisfying \eqref{cond} and let $\lam_1$ be the first eigenvalue of \eqref{eigen}. Assume that $T_g=\infty$ and that $H\prec G \prec\prec G^*$.  

Denote the Young function $A(t):=G^*\circ G^{-1}(t)$, being $G^*$ the critical exponent in the Orlicz-Sobolev embedding. Then,

\begin{itemize}
\item [(i)] if $w\in L^\infty(\Omega)$,  there exists a positive constant $c_1$ such that
\begin{equation} \label{t.peso.0}
\left[ G\left( \frac{c_1}{ G^{-1}\left(\frac{1}{\tau_A(\Omega)\|w\|_{L^\infty(\Omega)} }\right)   }  \right) \right]^{-1}\leq \lam_1 
\end{equation}
where 
$\tau_A(t)=|\Omega| A^{-1}(|\Omega|^{-1})$, $t>0$;

\item[(ii)] if $w\in L^{\tilde B}(\Omega)$ with $B$ a Young function such that $B\prec A$, there exists a positive constant $\hat c_1$ such that
\begin{equation} \label{t.peso.1}
\left[ G\left( \frac{\hat c_1}{ G^{-1}\left(\|w\|^{-1}_{\tilde B}\right)   }  \right) \right]^{-1} \leq \lam_1.
\end{equation}
\end{itemize}
\end{thm}

\begin{proof}
First, observe that since $A\circ G=G^*$ is a Young function, by \cite{KR}[Chapter I.1.2], $A$ is also a Young function.  Then, the H\"older's inequality for Young functions yields
\begin{equation} \label{thm.eq0}
1=\int_\Omega wH(|u|)\,dx \leq \hat k\int_\Omega wG(|u|)\,dx \leq  2\hat k  \|\chi_\Omega\|_{\tilde A} \|G(|u|)\|_A \|w\|_\infty
\end{equation}
since  $w$ is positive and $H\prec G$, for some fixed $\hat k>0$.

The first factor in the last expression can be estimated by using \cite[Example 3.6.9]{CFK} and \cite[Theorem 3.8.5]{CFK}:
\begin{equation} \label{thm.eq1}
\|\chi_\Omega\|_{\tilde A} \leq |\Omega| A^{-1}(|\Omega|^{-1}):=\tau_A(\Omega).
\end{equation}

On the other hand, observe that by using \eqref{cond1}
\begin{align*}
G^{-1}\left( \frac{G(|u|)}{ G(c\|u\|_{G^*})} \right) &= 
\frac{G^{-1}(G(c\|u\|_{G^*}))}{ c\|u\|_{G^*} } G^{-1}\left( \frac{G(|u|)}{ G(c\|u\|_{G^*})} \right)\\
&\leq
 \frac{1}{ \|u\|_{G^*} } G^{-1}\left(G(c\|u\|_{G^*})  \frac{G(|u|)}{G(c\|u\|_{G^*})}  \right) =
 \frac{|u|}{ \|u\|_{G^*} },
\end{align*}
which, in light of the definition of Luxemburg norm, gives
\begin{align*}
\int_\Omega A\left( \frac{G(|u|)}{G(c\|u\|_{G^*})} \right)\,dx &= 
\int_\Omega G^*\circ G^{-1}\left( \frac{G(|u|)}{ G(c\|u\|_{G^*})} \right)\,dx \leq 
\int_\Omega G^*\left(  \frac{|u|}{ \|u\|_{G^*} }  \right) \,dx=1
\end{align*}
since $G^*$ is increasing, which, again by definition of the norm,  leads to
$$
 \|G(|u|) \|_A \leq G(c\|u\|_{G^*}).
$$
This, together with  Proposition \ref{prop1} and the embedding of $W^{1,G}_0(\Omega)$ into $L^{G^*}(\Omega)$ given in Proposition \ref{embed.w} allows us to deduce that
\begin{align} \label{thm.eq2}
\begin{split}
\|G(|u|)\|_A &\leq G(c\| u \|_{G^*}) \leq G(c\kappa \|\nabla u\|_G) \leq G\left(
\frac{c\kappa \kappa_0}{G^{-1}(\lam^{-1}) } \|u\|_{H,w} \right)\\
&\leq G\left(
\frac{c\kappa \kappa_0}{G^{-1}(\lam^{-1}) }  \right),
\end{split}
\end{align}
where we have used Lemma \ref{equiv} for  $u\in \mathcal{A}$.

Finally, gathering \eqref{thm.eq0}, \eqref{thm.eq1} and \eqref{thm.eq2} we get
$$
\frac{1}{2\hat k \|w\|_\infty \tau_A(\Omega)}\leq G\left(
\frac{c\kappa\kappa_0}{G^{-1}(\lam^{-1}) }  \right),
$$
from where, using that $G^{-1}$ is an increasing function we arrive at
$$
G^{-1}\left( \frac{1}{2\hat k \|w\|_\infty \tau_A(\Omega)}\right) \leq 
\frac{c\kappa\kappa_0}{G^{-1}(\lam^{-1}) } 
$$
and \eqref{t.peso.0} follows.

On the other hand, if $w\in L^{\tilde B}(\Omega)$ with $B\prec A$, the H\"older's inequality for Young functions yields
\begin{equation*} 
1=\int_\Omega wH(|u|)\,dx \leq \hat k\int_\Omega wG(|u|)\,dx \leq  2 \hat k \|w\|_{\tilde B} \|G(|u|)\|_A,
\end{equation*}
since  $w$ is positive and $H\prec G$, for some fixed $\hat k>0$, and then, using \eqref{thm.eq2}, is obtained  \eqref{t.peso.1}.
\end{proof}

\begin{thm} \label{p.hardy}
Let $G$ and $H$ be Young functions  satisfying \eqref{cond} ordered as
\begin{equation} \label{cond.hardy}
t\int_{t_0}^t \frac{H(s)}{s^2}\,ds \leq G(kt)\quad \text{ for } t\geq t_0,
\end{equation}
for some constant $k>0$ and $t_0\geq 0$. Let $\lam_1$ be the first eigenvalue of \eqref{eigen} and $w\in L^\infty(\Omega)$, then
$$
\left[ G\left( \frac{ c_2 r_\Omega}{ H^{-1}(\|w\|_{L^\infty(\Omega)}^{-1})}\right) \right]^{-1} \leq \lam_1,
$$
where $c_2$ is a positive constant and $r_\Omega$ denotes the inner radius of $\Omega$.
\end{thm}

\begin{proof}
First, observe that taking $\delta=\|u\|_H/H^{-1}(C^{-1}\|w\|_{L^\infty(\Omega)}^{-1})$ in the definition of the Luxemburg norm together with \eqref{cond}  yields
\begin{align*}
\int_\Omega H\left( \frac{u}{\delta} \right) w(x)\,dx &\leq  
\int_\Omega H\left( \frac{|u|}{\|u\|_H} H^{-1}(C^{-1}\|w\|_{L^\infty(\Omega)}^{-1})\right) \|w\|_{L^\infty(\Omega)} \,dx\\  &\leq 
\int_\Omega H\left( \frac{|u|}{\|u\|_H}\right) \,dx =1,
\end{align*}
from where, again by definition of the Luxemburg norm, we get
\begin{equation} \label{des.h.1}
\|u\|_{H,w}\leq \frac{\|u\|_H}{H^{-1}(C^{-1}\|w\|_{L^\infty(\Omega)}^{-1})}.
\end{equation}

Now, in light of condition \eqref{cond.hardy},  by the Hardy inequality \cite[Theorem 1]{Cia} there exists a positive constant $c_H$ such that
$$
\Big\|\frac{u}{d(x,\partial \Omega)} \Big\|_H \leq c_H \|\nabla u\|_G.
$$
Since for any $x\in \Omega$ we have that $d(x,\partial \Omega)\leq r_\Omega$,  from the last expression we get
$$
\frac{\|u\|_H}{r_\Omega} \leq \Big\|\frac{u}{d(x,\partial\Omega)} \Big\|_H \leq c_H \|\nabla u\|_G \leq \frac{c_H \kappa_0}{G^{-1}(\lam_1^{-1}) } \|u\|_{H,w}
$$
where in the last inequality we have used Proposition \ref{prop1} and the $1-$homogeneity of the Luxemburg norm.

From the last inequality and \eqref{des.h.1}, and rearranging terms,  we   deduce that
$$
G^{-1}(\lam_1^{-1}) \leq \frac{c_H \kappa_0 r_\Omega}{H^{-1}(C^{-1}\|w\|_{L^\infty(\Omega)}^{-1})},
$$
and the result follows just by applying $G$ to both sides of the inequality.
\end{proof}

\begin{thm} \label{p.bigger.n}
Let $G$ and $H$ be Young functions  satisfying \eqref{cond} and assume that $T_g<\infty$. Let $\lam_1$ be the first eigenvalue of \eqref{eigen} and $w\in L^\infty(\Omega)$, then there are positive constants $c_3$ and $\hat c_3$ such that
\begin{equation} \label{p.bigger.n.1}
\left[G\left(  \frac{ c_3\sigma(r_\Omega)}{H^{-1}(\|w\|_{L^1(\Omega)}^{-1})}\right)\right]^{-1} \leq \lam_1
\end{equation}
and
\begin{equation} \label{p.bigger.n.2}
\left[G\left(\frac{\hat c_3\sigma(r_\Omega) \tau_H(\Omega)}{H^{-1}(\|w\|_{L^\infty(\Omega)}^{-1})}\right) \right]^{-1} \leq \lam_1
\end{equation}
where $\tau_H(\Omega):=|\Omega| (\tilde H)^{-1}(|\Omega|^{-1})$,  $r_\Omega$ denotes the inner radius of $\Omega$ and $\sigma$ is the modulus of continuity given in \eqref{modulo}.
\end{thm}

\begin{proof}
Let $u\in \mathcal{A}$ be an eigenfunction corresponding to $\lam_1$. By Proposition \ref{embed.w}, $u\in C^{0,\sigma(t)}(\overline \Omega)$. Since $u>0$ in $\Omega$ and $u=0$ on $\partial\Omega$, there exists $x_0\in \Omega$ such that $u(x_0)=\|u\|_{L^\infty(\overline \Omega)}$. Therefore, for any $y\in \partial\Omega$ we have that
$$
u(x_0)\leq \kappa \sigma(|x_0-y|) \| \nabla u\|_G
$$
which yields 
$$
u(x_0)\leq \kappa \sigma(r_\Omega) \| \nabla u\|_G.
$$
Gathering the last relation together with Proposition \ref{prop1} one gets that
\begin{equation} \label{estimacion}
u(x_0)  \leq  \sigma(r_\Omega) \frac{\kappa\kappa_0}{G^{-1}(\lam^{-1}) } \|u\|_{H,w}.
\end{equation}
From the fact that $\Phi_{H,w}(u)=1$, by using  Lemma \ref{equiv} we can write
\begin{equation} \label{dess33}
\|u\|_{H,w}=1 = H^{-1}\left(\int_\Omega wH(|u|)\,dx\right) \leq   H^{-1}\left(H(u(x_0)   \|w\|_{L^1(\Omega)} ) \right),
\end{equation}
and using condition \eqref{cond1} we obtain
\begin{align} \label{dess3}
\begin{split}
H^{-1}\left(H(u(x_0)\|w\|_{L^1(\Omega)} ) \right) &=  \frac{ H^{-1}(\|w\|_{L^1(\Omega)}^{-1})}{H^{-1}(\|w\|_{L^1(\Omega))}^{-1})} H^{-1}\left(H( u(x_0) \|w\|_{L^1(\Omega)} ) \right)\\
&\leq \frac{c}{H^{-1}(\|w\|_{L^1(\Omega)}^{-1})} H^{-1} \left(\|w\|_{L^1(\Omega)}^{-1} \|w\|_{L^1(\Omega)} H(u(x_0)) \right)\\
&=\frac{c u(x_0)}{H^{-1}(\|w\|_{L^1(\Omega)}^{-1})}.
\end{split}
\end{align}
Then,  \eqref{estimacion}, \eqref{dess33} and \eqref{dess3} give that
$$
u(x_0)\leq  \sigma(r_\Omega) \frac{c \kappa \kappa_0 }{G^{-1}(\lam_1^{-1}) } \frac{   u(x_0)}{H^{-1}(\|w\|_{L^1(\Omega)}^{-1})}.
$$
Rearranging terms we can rewrite the last inequality as
$$
G^{-1}(\lam^{-1}) \leq c \kappa \kappa_0 \frac{\sigma(r_\Omega)}{H^{-1}(\|w\|_{L^1(\Omega)}^{-1})}.
$$
and finally, applying $G$ to both sides we get \eqref{p.bigger.n.1}.

A slightly different bound can be obtained by observing that inequality \eqref{des.h.1} gives
\begin{align*}
\|u\|_{H,w} &\leq  \frac{\|u\|_H}{H^{-1}(\|w\|_{L^\infty(\Omega)}^{-1})}
\leq 
\frac{\|u\|_{L^\infty(\Omega)} \|\chi_\Omega\|_H}{{H^{-1}(\|w\|_{L^\infty(\Omega)}^{-1})}}\\
&\leq  \frac{u(x_0)}{{H^{-1}(\|w\|_{L^\infty(\Omega)}^{-1})}} |\Omega| (\tilde H)^{-1}(|\Omega|^{-1}):=\frac{u(x_0)}{{H^{-1}(\|w\|_{L^\infty(\Omega)}^{-1})}} \cdot \tau_H(\Omega)
\end{align*}
where we have used \cite[Example 3.6.9]{CFK} and \cite[Theorem 3.8.5]{CFK} to estimate the norm of a characteristic function, being $(\tilde H)^{-1}$  the inverse of the conjugated function of $H$. This  together with \eqref{estimacion} gives
$$
G^{-1}(\lam^{-1}) \leq  \frac{\kappa\kappa_0\sigma(r_\Omega)}{{H^{-1}(\|w\|_{L^\infty(\Omega)}^{-1})}}  \tau_H(\Omega),
$$
from where it follows \eqref{p.bigger.n.2}.
\end{proof}

\section{One-dimensional inequalities} \label{sec.1d}
In this section we deal with the eigenvalue problem \eqref{eigen} when $\Omega=(a,b)\subset \R$ is a bounded interval. In this case $W^{1,G}(\Omega) \subset L^H(\Omega)$ with compact inclusion for any couple of Young function satisfying \eqref{condp}. 
This allows us to prove a lower bound for $\lam_{1,\mu}$, $\mu\geq 1$,  without using Morrey type estimates.

As in the previous section, we  denote $\lam_1:=\lam_{1,1}$ and $\A:=\A_{1}$, and again, by Corollary \ref{coro}, a lower bound of $\lam_1$ is also a lower bound for $\mu \lam_{1,\mu}$ when $\mu\geq 1$.

\begin{thm} \label{teo1d-v1}

Let $G$ and $H$ be  Young functions satisfying   \eqref{cond}. Let $\lam_1$ be the first eigenvalue of \eqref{eigen} with $\Omega=(a,b)\subset \R$ a bounded interval. Assume that $w\in L^\infty(\Omega)$. Then it holds that
$$
\left[ G\left( 
2c\kappa_0(b-a) \frac{G^{-1}\left(\frac{1}{b-a}\right)}{H^{-1}\left(\|w\|_{L^1(\Omega)}^{-1}\right)}
 \right) \right]^{-1} \leq \lam_1.
$$
\end{thm}
\begin{proof}
Let $u\in \mathcal{A}$ be an eigenfunction corresponding to $\lam_1$. By Proposition \ref{prop.lam1} we have that $u>0$ in $\Omega$. Moreover, by \eqref{condp} and \cite[Theorem 3.17.1]{CFK} it follows that $W^{1,G}(\Omega)\subset W^{1,p^-}(\Omega)$, where $p^-$ is given in \eqref{condp}, so  $u$ equals a.e. a positive  absolutely continuous function. Hence, we can take $x_0\in \Omega$  such that $u(x_0)=\|u\|_{L^\infty(\Omega)}$. By using H\"older's inequality for Orlicz spaces we get
$$
u(x_0)\leq \int_a^b |u'(t)|\,dt \leq 2\|u'\|_G \|\chi_{(a,b)} \|_{\tilde G}.
$$
By \cite[Example 3.6.9]{CFK} and \cite[Theorem 3.8.5]{CFK} 
$$
\| \chi_{(a,b)} \|_{\tilde G} \leq  (b-a) G^{-1}\left( \frac{1}{b-a}\right).
$$
Then, the last two expression together with Proposition \ref{prop1} yield
$$
 u(x_0) \leq  2 (b-a) G^{-1}\left( \frac{1}{b-a}\right)  \frac{\kappa_0}{G^{-1}(\lam^{-1}) } \|u\|_{H,w},
$$
and invoking     \eqref{dess33} and \eqref{dess3} we get
$$
 u(x_0) \leq  2 (b-a) G^{-1}\left( \frac{1}{b-a}\right)  \frac{\kappa_0}{G^{-1}(\lam^{-1}) }\frac{c  u(x_0) }{H^{-1}\left(\|w\|_{L^1(\Omega)}^{-1}\right)}
$$
from where
$$
G^{-1}(\lam^{-1})\leq  2c\kappa_0(b-a) \frac{G^{-1}\left(\frac{1}{b-a}\right)}{H^{-1}\left(\|w\|_{L^1(\Omega)}^{-1}\right)}.
$$
Finally, the result follows by applying $G$ to both sides of the inequality.
\end{proof}

When we further assume $H\prec G$, following the ideas of \cite{Pi}, we have the following result which bounds by below to $\lam_{1,\mu}$ for any $\mu>0$ (observe that now $\mu$ is allowed to be less than 1).

\begin{thm} \label{teo1d-v2}
Under the assumptions of Theorem \ref{teo1d-v1}, if we further assume that $H\prec G$ with constant $k>0$, then
$$
\frac{b-a}{C k \|w\|_{L^1(\Omega)} G\left(\frac{k(b-a)}{2}\right)}\leq \lam_{1,\mu}
$$
for any $\mu>0$.
\end{thm}

\begin{proof}
Let $\lam_{1,\mu}$ with eigenfunction $u\in \A_\mu$ for any $\mu>0$. As in the proof of Theorem \ref{teo1d-v1}, let $x_0\in [a,b]$ be such that $u(x_0):=\|u\|_{L^\infty(\Omega)}$. We write
$$
2|u(x_0)|=\left| \int_a^{x_0} u'(x)\,dx \right|	 + \left| \int_{x_0}^b u'(x)\,dx \right|	\leq \int_a^b |u'(x)|\,dx,
$$
which yields
$$
\frac{2  u(x_0) }{b-a} \leq \frac{1}{b-a}\int_a^b |u'(x)|\,dx,
$$
and from Jensen's inequality and Remark \ref{rem.autov}
$$
G\left(\frac{2  u(x_0) }{b-a} \right) \leq \frac{1}{b-a}\int_a^b G(|u'|)\,dx = \frac{\lam_{1,\mu}}{b-a} \int_a^b wH(u)\,dx
$$
from where, since $H\prec G$, for some fixed $k>0$
\begin{align*}
G\left(\frac{2  u(x_0) }{b-a} \right) &\leq  \frac{\lam_{1,\mu}}{b-a} \int_a^b wG(ku)\,dx
\leq \frac{\lam_{1,\mu}}{b-a}  G(k u(x_0) ) \|w\|_{L^1(\Omega)}.
\end{align*}
Observe that, using \eqref{cond} we get
\begin{equation*}
G(ku(x_0)) = G\left(ku(x_0) \frac{2}{b-a} \frac{b-a}{2}\right)  \leq C G\left( \frac{2 u(x_0) }{b-a} \right) G\left(\frac{k(b-a)}{2}\right).
\end{equation*}
From the last two inequalities we get
$$
1\leq \frac{\lam_{1,\mu}}{b-a} \|w\|_{L^1(\Omega)} k C   G\left(\frac{k(b-a)}{2}\right)
$$
and the result follows.
\end{proof}

\section{Final remarks  } \label{sec.rem}

Our results are based in the fact that $\lam_{1,1}$ is always a lower than $\lam_{1,\mu}$ when $\mu\geq 1$. Hence it is enough with looking for lower bounds of that  quantity, and it is possible due to the equivalence between norm  and modular when $\mu=1$. A natural question is to determine if our results still true when $0<\mu<1$. Moreover, it would be interesting to obtain similar result without assuming the $\Delta'$ condition on $G$ and $H$.

Finally, in \cite{FBS} a fractional version of the $g-$Laplacian was introduced. Eigenvalues of that operator were recently studied in \cite{BS, S, SaVi}. Then, a natural question is to analyze whether an extension of the results of this manuscript can be formulated for the nonlocal  version of problem  \eqref{eigen.intro}. The case of powers, i.e., $G(t)=H(t)=t^p$ was dealt in \cite{JKS}.

\section*{Acknowledgments}
The author wants to thank to J. Fern\'andez Bonder and H. Vivas for valuable comments on the subject.

\end{document}